\documentclass[preprint]{imsart}
\RequirePackage[OT1]{fontenc}
\RequirePackage{amsthm,amsmath}
\RequirePackage[colorlinks,citecolor=blue,urlcolor=blue,backref=true,linktocpage=true]{hyperref}

\startlocaldefs
\numberwithin{equation}{section}
\theoremstyle{plain}
\endlocaldefs

\usepackage{amsmath}
\usepackage{amsmath, amsthm, amssymb}
\usepackage{accents}
\usepackage{graphicx}
\usepackage{color}
\usepackage{enumerate}
\usepackage{textcomp}
\usepackage{graphicx}
\usepackage{graphicx,hyperref}
\usepackage{hyperref}
\usepackage{natbib}
\usepackage{fancyhdr}
\usepackage[left=1.5in, right=1.5in]{geometry}

\newtheorem{thm}{Theorem}[section]
\newtheorem{lem}{Lemma}[section]

\theoremstyle{definition}
\newtheorem{definition}{Definition}[section]

\newtheorem{assumption}{Assumption}[section]

\numberwithin{equation}{section}

\begin{document}

\begin{frontmatter}
\title{Posterior Concentration Properties of a General Class of Shrinkage Priors around Nearly Black Vectors}
\runtitle{Posterior Contraction Of Shrinkage Priors}

\begin{aug}
\author{\fnms{Prasenjit} \snm{Ghosh}\ead[label=e1]{prasenjit$\_$r@isical.ac.in}}
\and
\author{\fnms{Arijit} \snm{Chakrabarti}\ead[label=e2]{arc@isical.ac.in}}

\runauthor{P. Ghosh and A. Chakrabarti}

\affiliation{Indian Statistical Institute}

\address{Applied Statistics Unit, Indian Statistical Institute, Kolkata, India.\\
\printead{e1},
\printead*{e2}}

\end{aug}

\begin{abstract}
We consider the problem of estimating a high-dimensional multivariate normal mean vector when it is sparse in the sense of being nearly black. Optimality of Bayes estimates and posterior concentration properties corresponding to a very general class of continuous shrinkage priors on the unknown mean vector are studied in this work. The class of priors considered is rich enough to include a wide variety of heavy tailed prior distributions including the horseshoe which are in extensive use in sparse high-dimensional problems. In particular, the three parameter beta normal mixtures, the generalized double Pareto priors, the inverse gamma priors and the normal-exponential-gamma priors fall inside this class. We work under the frequentist setting where the data is generated according to a multivariate normal distribution with a fixed unknown mean vector. Under the assumption that the number of non-zero components of the mean vector is known, we show that the Bayes estimators corresponding to this general class of priors attain the minimax risk (possibly up to a multiplicative constant) corresponding to the $l_2$ loss. It is further shown that the posterior distributions arising out of this general class of one-group priors contract around the underlying true mean vector at the minimax optimal rate for appropriate choice of the global shrinkage parameter depending on the proportion of non-zero components of the underlying mean vector. Moreover, we provide conditions on the choice of the global shrinkage parameter for which these posterior distributions contract around the corresponding Bayes estimates at least as fast as the minimax risk with respect to the $l_2$ norm. We also provide a lower bound to the total posterior variance for an important subclass of this general class of shrinkage priors that includes the generalized double Pareto priors with shape parameter $\alpha=1$ and the three parameter beta normal mixtures with parameters $a=0.5$ and $b>0$ (including the horseshoe) in particular. The present work is inspired by the recent work of \citet{PKV2014} on the posterior contraction properties of the horseshoe prior under the present set-up. We extend their results for this general class of priors and come up with novel unifying proofs which work for a very broad class of one-group continuous shrinkage priors. This work shows that the general scheme of arguments in \citet{PKV2014} can be used in greater generality.
\end{abstract}

%
\end{frontmatter}

\section{Introduction}
With rapid advancements in modern technology and computing facilities, high throughput data have become common place in real life problems across diverse scientific fields such as genomics, biology, medicine, cosmology, finance, economics and climate studies. As a result inferential problems involving a large number of unknown parameters are coming to the fore. Problems where the number of unknown parameters grows as least as fast as the number of observations are typically called high-dimensional. In such problems, often times it is also true that only a few of these parameters are of real importance. For example, in a high dimensional regression problem, it is often true that the proportion of non-zero regressors or regressors with large magnitude is quite small compared to the total number of candidate regressors. This is called the phenomenon of sparsity. A common Bayesian approach to model sparse high-dimensional data is to use a two-component point mass mixture prior for the parameters and they put a positive mass at zero (to induce sparsity) and a heavy tailed continuous distribution (to identify the non-zero coefficients). These are also referred to as ``spike and slab priors'' or two-groups priors. This is a very natural way of modeling data of this kind from a Bayesian view point. See \citet{JS2004} and \citet{Efron2004} in this context.\newline

Use of the two-groups prior, although very natural, poses a very daunting task computationally. Note that the cardinality of the model space becomes $2^n$ where $n$ is the number of parameters involved, and even for moderately large $n$ like 50, it is practically impossible to study posterior probabilities of the different models. Sometimes it is also possible that most of the parameters are very close to zero, but not exactly equal to zero. So in such a case a continuous prior may be able to capture sparsity in a more flexible manner. Due to these reasons, significant efforts have gone into modeling sparse high-dimensional data in recent times through hierarchical one-group continuous priors, which are also called one-group shrinkage priors. Bayesian analysis is computationally much more tractable than the two-group prior in such cases and easily implementable through standard MCMC techniques. But more importantly, these priors are suitable to capture sparsity since they accord a significant chunk of probability around zero while they have tails which are heavy enough to ensure a priori large probabilities for large parameter values. In general, such priors are expressed as multivariate scale-mixtures of normals that mix over two levels of parameters appearing in the scale, referred to as a ``global'' shrinkage parameter and a ``local'' shrinkage parameter. While the global shrinkage parameter accounts for the overall sparsity in the data by shrinking the noise observations to the origin, the local shrinkage parameters are helpful in detecting the obvious signals by leaving the large observations mostly unshrunk.\newline

A great variety of one-group shrinkage priors have appeared in the literature over the years. Notable early examples are the $t$-prior in \citet{Tipping2001}, the double-exponential prior in \citet{PC2008} and \citet{Hans2009} and the normal-exponential-gamma priors in \citet{GB2005}. Very recently \citet{CPS2009, CPS2010} introduced the horseshoe prior, which has very appealing properties. Subsequently, many other one-group priors have been proposed in the literature, e.g, in \citet{PS2011, PS2012}, \citet{ADC2011}, \citet{ADL2012} and \citet{GB2010, GB2012, GB2013}. The class of ``three parameter beta normal'' mixture priors was introduced in \citet{ADC2011}, while the ``generalized double Pareto'' class of priors was introduced by \citet{ADL2012}. The three parameter beta normal mixture family of priors encompasses among others the horseshoe, the Strawderman-Berger and the normal-exponential-gamma priors. Very recently, a different class of one-group priors named Dirichlet-Laplace (DL) priors have been introduced in \citet{BPPD2014}. They investigated its various theoretical properties and demonstrated its good performances through extensive simulations.\newline

As commented in \citet{CV2012}, the Bayesian approach to sparsity is not driven by the ultimate goal of producing estimators that attain the minimax rate or for that matter posterior distributions with rate of contraction same as the minimax rate. However, for theoretical investigations, minimax rate can be taken as a benchmark and this is a motivation to study this kind of optimality properties for the Bayesian approach to sparsity. In an important article, \citet{JS2004} focused on the case where a two-groups prior is used to model the mean parameters. They showed that if the unknown proportion of non-zero means is estimated by marginal maximum likelihood and a coordinate-wise posterior median estimate is used, the resulting estimator attains the minimax rate with respect to $l_q$ loss, $q \in (0, 2]$. In \citet{CV2012}, the full Bayes approach was studied where they found conditions on the two-groups prior that ensure contraction of the posterior distribution at the minimax rate.\newline

In recent times, researchers have started to investigate various theoretical properties of Bayes estimates and testing rules based on one-group continuous shrinkage priors. Amongst various such one-group priors, the horseshoe prior has acquired a prominent place in the Bayesian literature and it has been used extensively in inferential problems involving sparsity. \citet{CPS2010} have theoretically showed good performance of the horseshoe estimator (the Bayes estimate corresponding to the horseshoe prior) in terms of the Kullbuck-Leibler risk when the true mean is zero. \citet{DG2013} showed a near oracle optimality property of multiple testing rules based on the horseshoe estimator in the context of multiple testing. \citet{GTGC2015} extended their work by theoretically showing that the multiple testing rules based on a general class of tail robust shrinkage priors enjoy similar optimality properties as the horseshoe. They also considered an empirical Bayes approach by plugging in an estimate of the global variance component into the definition of such multiple testing rules, and showed that the resulting empirical Bayes procedure enjoys similar optimality property. This general class of shrinkage priors is rich enough to include among others, the three parameter beta normal priors, the generalized double Pareto priors, the inverse gamma priors, and the normal-exponential-gamma priors, the horseshoe prior and the Strawderman-Berger prior, in particular.\newline

In an important recent article, \citet{PKV2014} showed that for the problem of estimation of a sparse normal mean vector, the horseshoe estimator asymptotically achieves the minimax risk with respect to the $l_2$ loss, possibly up to a multiplicative constant and the corresponding posterior distribution contracts around the true mean vector at this rate. Moreover, they showed that the corresponding posterior distribution contracts at least as fast as the minimax rate around the posterior mean. This was shown assuming that the number of non-zero means is known and the global shrinkage parameter tends to zero at an appropriate rate as the dimension grows to infinity. They also provide conditions under which the horseshoe estimator combined with an empirical Bayes estimate of the global variance component still attains the minimax quadratic risk even when the number of non-zero means is unknown. In a beautiful recent article, \citet{BPPD2014} showed that for the estimation of a sparse multivariate normal mean vector, under the quadratic risk function, the posterior arising from the Dirichlet-Laplace prior attains a minimax optimal rate of posterior contraction, that is, the corresponding posterior distribution contracts at the minimax rate for an appropriate choice of the underlying Dirichlet concentration parameter. See also \citet{BRT2009} for the minimax risk properties of the Lasso estimator which is the least squares estimator of the regression coefficients with an $L_1$ constraint on the regression coefficients. It was later shown by \citet{CSHV2014} that the corresponding entire posterior distribution contracts at a much slower rate, thus indicating an inadequate measure of uncertainty in the estimate.\newline

A natural question to ask, and also posed in section 6 of \citet{PKV2014}, is what aspects of the shrinkage priors are essential towards obtaining optimal posterior concentration properties as obtained for the case of the horseshoe prior. As mentioned earlier, in the context of simultaneous testing of a large number of independent normal means, \citet{GTGC2015} considered a general class of heavy-tailed shrinkage priors and showed certain optimality properties of the multiple testing rules induced by the corresponding Bayes estimates. \citet{PS2011} suggested that in sparse problems, one should choose the prior distribution corresponding to the local shrinkage parameter to be appropriately heavy-tailed so that large signals can escape the ``gravitational pull'' of the corresponding global variance component and are almost left unshrunk which is essential for the recovery of large signals when the data is sparse. It is to be mentioned in this context that priors with exponential or lighter tails, such as the Laplace or the double-exponential prior and the normal prior, fail to meet this condition.\newline

Motivated by this, we consider in this article, the problem of estimating a sparse multivariate normal mean vector based on a very general class of ``tail-robust'' one-group priors that is rich enough to include a wide variety of shrinkage priors, such as, the three parameter beta normal mixtures (which generalizes the horseshoe prior in particular), the generalized double Pareto prior, the inverse-gamma priors, the half-t priors and many more. We work under the framework when the number of non-zero components of the unknown mean vector is assumed to be known and the global shrinkage parameter is treated as a tuning parameter that we are free to choose and which tends to zero at an appropriate rate as the dimension grows to infinity. It is shown that when the underlying multivariate normal mean vector is sparse in the nearly-black sense, the Bayes estimates corresponding to this general class of priors asymptotically attain the minimax-quadratic risk, possibly up to a multiplicative factor and the corresponding posterior distributions contract around the true mean vector at the minimax optimal rate. Moreover, we provide conditions on the choice of the global shrinkage parameter for which these posterior distributions contract around the corresponding Bayes estimates at least as fast as the minimax risk with respect to the $l_2$ norm. We also provide a lower bound to the corresponding posterior variance for an important subclass of this general class of shrinkage priors that include the generalized double Pareto priors with shape parameter $\alpha=1$, the three parameter beta normal mixtures with parameters $a=\frac{1}{2}$ and $b>0$ (including the horseshoe in particular), the inverse gamma prior with shape parameter $(\alpha=\frac{1}{2})$ and many other shrinkage priors. This lower bound helps us to deduce further insights about the posterior spread of the aforesaid sub-family of tail robust priors and also regarding the choice of the global shrinkage component. An important contribution of our theoretical investigation is showing that shrinkage priors which are appropriately heavy-tailed (to be defined in Section 2), are able to attain the minimax optimal rate of contraction, provided that the global tuning parameter is carefully chosen. We provide a general unifying argument that works for this general class under consideration and thus extends the work of \citet{PKV2014}.\newline

We organize the paper as follows. In Section 2, we describe the problem and the general class of shrinkage priors under consideration. Section 3 contains the main theoretical results, that estimators arising out of this general class of shrinkage priors attain the minimax quadratic risk up to some multiplicative constant and that the corresponding posterior distribution results in a minimax optimal rate of posterior contraction around the truth. Section 4 contains some concluding remarks. Proofs of the main theorems and other supporting results essential for their derivation are given in the Appendix.

\subsection*{Notations and Definition}
In this paper, we adopted the same convention of notation used in \citet{PKV2014}. Let $\{A_n\}$ and $\{B_n\}$ be two sequences of positive real numbers indexed by $n$. We write $A_n\asymp B_n$ to denote $0 < \lim_{n \rightarrow \infty} \inf_{n} \frac{A_n}{B_n} \leqslant \lim_{n \rightarrow \infty} \sup_{n} \frac{A_n}{B_n} < \infty$ and $A_n \lesssim B_n $ to denote that there exists some $c > 0$ independent of $n$ such that $A_n \leqslant cB_n$.

\begin{definition}\label{DEFN_SVF}
A positive measurable function $L$ defined over some $(A,\infty)$, $A \geqslant 0$, is said to be slowly varying or is said to vary slowly (in Karamata's sense) if for every fixed $ \alpha > 0$, $L(\alpha x)\sim L(x)$ as $x \rightarrow \infty$.
\end{definition}

A simple sufficient condition for a positive measurable function $L$ to be slowly varying is that $L(x) \rightarrow c \mbox{ as } x \rightarrow \infty$, for some $c \in (0,\infty)$.



\section{A General Class of Tail Robust Shrinkage Priors}

Suppose that we observe an n-component random observation $(X_1,\cdots,X_n) \in \mathbb{R}^n$, such that
\begin{equation}
 X_i = \theta_i + \epsilon_i  \mbox{  for } i =1,\cdots,n, \label{TRUE_MODEL}
\end{equation}
where the unknown parameters $\theta_1, \cdots, \theta_n $ denote the effects under investigation and $\epsilon=(\epsilon_1,\cdots,\epsilon_n) \sim N_{n}(0, I_n)$.\newline

Let $l_0[p_n]$ denote the subset of $\mathbb{R}^n$ given by,
\begin{equation}
l_0[p_n] = \{\theta \in \mathbb{R}^n : \#(1 \leqslant j \leqslant n : \theta_j \neq 0) \leqslant p_n\}. 
\end{equation}
 
Suppose we want to estimate the true mean vector $\theta_0=(\theta_{01},\cdots,\theta_{0n})$ when $\theta_0$ is known to be sparse in the ``nearly black sense'', that is, $\theta_0 \in l_0[p_n] $ with $p_n=o(n)$ as $n \rightarrow \infty$. The corresponding minimax rate with respect to the $l_2$-norm for estimating $\theta_0$ is given by (see \citet{DJHS1992}),
\begin{eqnarray}
 \inf_{\hat{\theta}}\sup_{\theta_{0}\in l_{0}[p_{n}]}E_{\theta_{0}}||\hat{\theta}-\theta_{0}||^2=2p_n\log\big(\frac{n}{p_n}\big)(1+o(1)),\mbox{ as } n\rightarrow\infty\label{MINIMAX_RATE}. 
 \end{eqnarray}
 In (\ref{MINIMAX_RATE}) above and throughout this paper $E_{\theta_{0}}$ denotes an expectation with respect to the $N_{n}(\theta_0, I_n)$ distribution. Our goal is to obtain an estimate of $\theta_0$ from a Bayesian view point with some good theoretical properties. As stated already in the introduction that a natural Bayesian approach to model (\ref{TRUE_MODEL}) is to use a two-component point mass mixture prior for the $\theta_i$'s, given by,
\begin{equation}\label{TWO_GROUP_THETA}
 \theta_i \stackrel{i. i. d.}{\sim} (1-\pi)\delta_{\{0\}} + \pi\cdot f, \mbox{ $i=1,\cdots,m.$}
 \end{equation} 
where $\delta_{\{0\}}$ denotes the distribution having probability mass 1 at the point 0, and $f$ denotes an absolutely continuous distribution over $\mathbb{R}$. See \citet{MB1988} and \citet{JS2004} in this context. It is usually recommended to choose a heavy tailed absolutely continuous distribution $f$ over $\mathbb{R}$ so that large observations can be recovered with higher degree of accuracy. \citet{JS2004} used a $t$ distribution in this context and used an empirical Bayes approach in order to estimate the unknown mixing proportion $\pi$ via the method of marginal maximum likelihood and showed that if the coordinate-wise posterior median estimate is used, the resulting estimator of $\theta_0$ attains the minimax rate with respect to the $l_q$ loss, $q \in (0, 2]$. \citet{CV2012} studied the full Bayes approach where they found conditions on the two-groups prior that ensure contraction of the posterior distribution at the minimax rate. A detailed list of other empirical Bayes approaches to the two-group model can be found in \citet{CV2012}, \citet{Efron2008} \citet{JZ2009}, \citet{YL2005} and references therein.\newline

 As already mentioned in the introduction that although the two groups prior (\ref{TWO_GROUP_THETA}) is considered to be the most natural formulation for handling sparsity from a Bayesian view point, it offers a daunting computational challenge in high dimensional problems because of the enormously large model space (in this case it is $2^n$). Due to this reason, the one-group formulation to model sparse data has received considerable attention from researchers over the years, mostly due to the ease of their computational tractability. \citet{PS2011} showed that almost all such shrinkage priors can be expressed as multivariate scale-mixture of normals which makes the computation based on these one-group shrinkage priors much easier compared to the corresponding two-group formulation. Standard Markov-chain Monte Carlo techniques are available in the Bayesian literature for the computation of the corresponding Bayes estimates of the underlying model parameters. In this article, we consider Bayes estimators based on a general class of one-group shrinkage priors given through the following hierarchical one-group formulation:
  \begin{align*}
  \begin{array}{ll}\label{ONE_GRP_PRIOR}
  X_i|\theta_i &\sim \mbox{ } N(\theta_i,1), \mbox{ independently for } i=1,\cdots,m\\
  \theta_{i}|(\lambda_i^2, \tau^2) &\sim \mbox{ }N(0,\lambda_i^2\tau^2), \mbox{ independently for } i=1,\cdots,m \nonumber\\
  \lambda_i^2 &\sim \mbox{ } \pi(\lambda_i^2), \mbox{ independently for } i=1,\cdots,m
  \end{array}
 \end{align*} 
 with $\pi(\lambda_i^2)$ being given by,
 \begin{equation}
 \label{LAMBDA_PRIOR}
  \pi(\lambda_i^2) =K (\lambda_i^2)^{-a-1}L(\lambda_i^2),
 \end{equation}
 where $K \in (0,\infty)$ is the constant of proportionality, $a$ is a positive real number and $L:(0,\infty) \rightarrow (0,\infty)$ is a measurable, non-constant slowly varying function. For the theoretical development in this paper, we assume that the function $L(\cdot)$ in (\ref{LAMBDA_PRIOR}) satisfies the following:
\begin{assumption}\label{ASSUMPTION_LAMBDA_PRIOR}
\qquad
 \begin{enumerate}
 \item  $\lim_{t \rightarrow \infty}L(t) \in (0,\infty),$ that is, there some exists $c_{0}(>0)$ such that $L(t) \geqslant c_{0}$ for all $t\geqslant t_{0}$, for some $t_{0}>0$, which depends on both $L$ and $c_{0}$. 
 
 \item There exists some $0 < M < \infty $ such that $\sup_{t\in(0,\infty)}L(t) \leqslant M.$
\end{enumerate}
\end{assumption}

Each $\lambda_i^2$ is referred to as a local shrinkage parameter and the parameter $\tau^2$ is called the global shrinkage parameter. For the theoretical treatment of this paper, we assume the global shrinkage parameter $\tau$ to be known. We would like to mention here that a very broad class of one-group shrinkage priors actually fall inside this above general class. For example, it can be easily seen that the celebrated horseshoe prior is a member of this general class under study by simply taking $a=0.5$ and $L(t)= t/(1+t)$ in (\ref{LAMBDA_PRIOR}) satisfying both the conditions of Assumption \ref{ASSUMPTION_LAMBDA_PRIOR}. \citet{GTGC2015} observed that the three parameter beta normal mixtures (which include the horseshoe and the normal-exponential-gamma priors as special cases) and the generalized double Pareto priors can be expressed in the above general form by showing that the corresponding prior distribution of the local shrinkage parameters can be written in the form given in (\ref{LAMBDA_PRIOR}) with the corresponding $L(\cdot)$ satisfying Assumption \ref{ASSUMPTION_LAMBDA_PRIOR}. It is easy to verify that some other well known shrinkage priors such as the families of inverse-gamma priors and the half-t priors are also covered by this general class of prior distributions under consideration. We would like to mention in this context that the above general class exclude priors such as the double-exponential or Laplace prior or the normal prior which have exponential or lighter tails.\newline

From Theorem 1 of \citet{PS2011} it follows that the above general class of one-group priors will be ``tail-robust'' in the sense that for any given $\tau > 0$, $E(\theta_i|X_i,\tau^2) \approx X_i,$ for large $X_i$'s, which means for such priors large observations will be almost left unshrunk even when the global shrinkage parameter $\tau$ is too small. We shall elucidate this fact in some greater detail in the Appendix. It was suggested in \citet{PS2011} that the global shrinkage parameter $\tau$ should be small so that small $X_i$'s or the noise observations can be shrunk towards the origin while the prior distribution of the local shrinkage parameters $\lambda_i^2$ should have heavy tails so that large signals can escape the effect of $\tau^2$ and almost remain unshrunk. Thus (\ref{LAMBDA_PRIOR}) should result in a prior distribution for the $\theta_i$'s which has a high concentration of mass near the origin but have thick tails at the extremes to accommodate large signals. \citet{PS2011} also showed that for priors having exponential or lighter tails, such as the Laplace or the double-exponential prior, even the large $X_i$'s will always be shrunk towards the origin by some non-diminishing amount for small values of $\tau$, which is certainly not desirable for the recovery of large signals in sparse situations.\newline

Now for a general global-local scale mixture of normals we have,
\begin{equation}
 \theta_{i}|(X_i, \lambda_i^2, \tau^2) \sim N((1-\kappa_i)X_i,(1-\kappa_i)), \mbox{ } \kappa_i=1/(1+\lambda_i^2\tau^2),\nonumber
\end{equation}
independently for $i=1,\cdots,m$, so that for each $i$, the posterior mean of $\theta_i$ is given by,
\begin{equation} \label{POST_MEAN_OG1}
 E(\theta_{i}|X_i, \lambda_i^2, \tau^2)=(1-\kappa_i)X_i, .
\end{equation} 
Next, using the iterated expectation formula it follows that,
   \begin{equation} \label{POST_MEAN_OG2}
   E(\theta_{i}|X_i,\tau^2)=(1-E(\kappa_i|X_i,\tau^2))X_i.
   \end{equation}

The resulting posterior mean $E(\theta|X,\tau)=(E(\theta_{1}|X_1,\tau^2),\cdots,E(\theta_{m}|X_m,\tau^2))$ will be the Bayes estimator arising out of the general class of shrinkage priors (\ref{LAMBDA_PRIOR}) and will be denoted by $T_{\tau}(X)$. For notational convenience, we shall denote $E(\theta_{i}|X_i,\tau^2)$ by $T_{\tau}(X_i)$. It will be shown in the next section that when $\theta_0$ is sparse in the ``nearly black sense'' and $a \in [\frac{1}{2},1)$, the estimator $T_{\tau}(X)$ of $\theta_0$ will asymptotically attain the minimax rate in (\ref{MINIMAX_RATE}) up to some multiplicative constant and that the posterior distribution contracts at least around the true $\theta_0$ at the minimax rate for suitably chosen $\tau$ depending on $\frac{p_n}{n}$.

\section{Theoretical Results}
In this section, we present the theoretical results involving the mean square error for the Bayes estimates arising out of the general class of one-group shrinkage priors under study, with $a\in [\frac{1}{2},1)$, and the spread of the corresponding posterior distributions. It is assumed that the number of non-zero components $p_n$ of the unknown mean vector is known. Theorem \ref{THM_3.1} gives an upper bound on the mean square error for the Bayes estimates arising out of the general class one-group priors under consideration. Using this upper bound, it follows that for various choices of the global shrinkage parameter $\tau$, depending on the proportion of non-zero means $\frac{p_n}{n}$, the aforesaid Bayes estimates attain the minimax risk with respect to the $l_2$-norm, possibly up to a multiplicative constant. Theorem \ref{THM_3.2} gives an upper bound to the total posterior variance for the one-group priors under study. Theorem \ref{THM_3.3} provides an upper bound to the rate of contraction around the true mean vector for the posterior distributions corresponding to the aforesaid class of priors, which shows minimax optimal contraction property around the true mean vector for these posterior distributions. However, the same may not be true for contraction around the corresponding Bayes estimates. Theorem \ref{THM_3.4} provides condition on the choice of the global shrinkage parameter $\tau$ that ensures that the resulting posterior distributions contract around the corresponding Bayes estimates at least as fast as the minimax-quadratic risk. Theorem \ref{THM_3.5} provides a lower bound to the total posterior variance for an important subclass of this general class of shrinkage priors that gives more insight about the spread of the posterior distribution around these estimators for various choices of $\tau$. We provide some novel unifying arguments that work for a broad class of one-group shrinkage priors. However, we follow the broad architecture of the proofs of the main theorems of \cite{PKV2014}. Lemmas \ref{LEMMA_MOST_IMP} - \ref{LEMMA_IMPORTANT_INTEGRAL}, given in the Appendix, on which Theorems \ref{THM_3.1} - \ref{THM_3.3} crucially hinge upon, are completely independent of the work of \citet{PKV2014}. However, proofs of Lemma \ref{LEMMA_POST_VAR_LOWER_BOUND} and Theorem \ref{THM_3.5} have been derived following some key arguments of \citet{PKV2014}. This shows that the general scheme of arguments in \citet{PKV2014} can be used in greater generality.

\begin{thm}\label{THM_3.1}
 Suppose $X \sim \mathcal{N}_{n}(\theta_{0},I_{n})$. Then the estimator $T_{\tau}(x)$ based on the general class of shrinkage priors (\ref{LAMBDA_PRIOR}), with $\frac{1}{2} \leqslant a < 1$, satisfies
 \begin{equation}\label{MSE_BAYES_ESTIMATE}
 \sup_{\theta_{0}\in l_{0}[p_{n}]}E_{\theta_{0}}||T_{\tau}(X)-\theta_{0}||^2
 \lesssim p_n\log\big(\frac{1}{\tau^{2a}}\big)+(n-p_n)\tau^{2a}\sqrt{\log\big(\frac{1}{\tau^{2a}}\big)}
 \end{equation}
 if $\tau \rightarrow 0$, as $n\rightarrow\infty$, $p_n\rightarrow \infty$ and $p_n=o(n)$.
\end{thm}
\begin{proof}
 See Appendix.
\end{proof}

An important consequence of Theorem \ref{THM_3.1} is that, for $\frac{1}{2} \leqslant a < 1$, the Bayes estimators based on the general class of one-group priors under study, attain the minimax quadratic risk up to some multiplicative factor for various choices of the global shrinkage parameter $\tau$. For example, by taking $\tau=\big(\frac{p_n}{n}\big)^{\alpha}$, $\alpha \geqslant 1$, in the upper bound in (\ref{MSE_BAYES_ESTIMATE}), it follows that the worst case $l_2$-risk for the Bayes estimates as considered in Theorem \ref{THM_3.1} can at most be of the order of $p_n\log\big(\frac{n}{p_n}\big)$ up to a multiplicative factor, the factor being $4a\rho^2\alpha$, while it is always bounded below the minimax risk $2p_n\log(n/p_n)(1+o(1))$ as $n\rightarrow\infty$. In particular,
 \begin{equation}\label{BAYES_ESTIMATE_MINIMAXITY}
 \sup_{\theta_{0}\in l_{0}[p_{n}]}E_{\theta_{0}}||T_{\tau}(X)-\theta_{0}||^2
 \asymp p_n\log\big(\frac{n}{p_n}\big).
 \end{equation}
It should also be noted that there are other choices of $\tau$, such as, $\tau=\frac{p_n}{n}\sqrt{\log(n/p_n)}$, which satisfy (\ref{BAYES_ESTIMATE_MINIMAXITY}) also. It is worth noting in this context that for priors like horseshoe or standard double Pareto for which $a=0.5$, the aforesaid multiplicative factor becomes $2\rho^2\alpha$ which is an improvement over the the corresponding multiplicative factor $4\alpha$ as obtained by \citet{PKV2014} for the horseshoe prior, provided $\rho>1$ is chosen in such a way that $\rho^2<2$. Therefore, we see that the Bayes estimators based on this general class of priors, with $a\in[\frac{1}{2},1)$, performs well as a point estimator, since each of them attains the minimax risk (\ref{MINIMAX_RATE}) up to some multiplicative constant, provided the global shrinkage parameter is carefully chosen. One possible explanation for such good performance of these priors is their ability to squelch the noise observations back to the origin, while leaving the large observations almost unshrunk, for appropriately chosen $\tau$. Moreover, smaller values of $a$ typically result in a prior distribution with heavier tails. For example, for the inverted beta families, $a=0.5$ yields Cauchy like tails. See \citet{PS2012} in this context. Similar discussion on how the choice of the parameter $a$ controls the tail behavior of the generalized double Pareto priors can also be found in \citet{ADL2012}. Recall that for the generalized double Pareto priors, we have $a=\alpha/2$, where $\alpha$ denotes the corresponding shape parameter (see \cite{GTGC2015}). \citet{ADL2012} recommended using the standard double Pareto distribution as a default prior specification which has Cauchy like tail and for which one has $a=0.5$. Thus, Theorem \ref{THM_3.1} extends the asymptotic minimaxity property of the horseshoe estimator obtained by \citet{PKV2014} over a large class of heavy tailed one-group shrinkage priors and can be considered as an important theoretical justification for the use of such priors when $a\in[\frac{1}{2},1)$.\newline

The next theorem gives an upper bound related to the total posterior variance corresponding to our general class of heavy tailed shrinkage priors.

\begin{thm}\label{THM_3.2}
 Suppose $X \sim \mathcal{N}_{n}(\theta_{0},I_{n})$. Then the variance of the posterior distribution corresponding to the general class of shrinkage priors (\ref{LAMBDA_PRIOR}), with $\frac{1}{2} \leqslant a < 1$, satisfies
 \begin{equation}\label{EXPECT_POST_VAR_UB}
 \sup_{\theta_{0}\in l_{0}[p_{n}]}E_{\theta_{0}} \sum_{i=1}^{n} Var(\theta_{0i}|X_i) \lesssim
  p_n+(n-p_n)\tau^{2a}\sqrt{\log\big(\frac{1}{\tau^{2a}}\big)}.
 \end{equation}
 
if $\tau \rightarrow 0$, as $n\rightarrow\infty$, $p_n\rightarrow \infty$ and $p_n=o(n)$.
\end{thm}
\begin{proof}
 See Appendix.
\end{proof}

It should be noted that, for the general class of priors under study, with $a\in[\frac{1}{2},1)$, the upper bound in (\ref{EXPECT_POST_VAR_UB}) is sharper as compared to that given in Theorem 3.2 of \citet{PKV2014} for the horseshoe prior. Theorems \ref{THM_3.1} and \ref{THM_3.2} together allow us to find a sharp upper bound on the rate of contraction of the full posterior distribution around the underlying true mean vector as given by the following theorem.

\begin{thm}\label{THM_3.3}
 Under the assumptions of Theorem \ref{THM_3.2}, if $\tau=\big(\frac{p_n}{n}\big)^{\alpha}$, with $\alpha\geqslant1$, then
 \begin{equation}\label{OPT_POST_CON_ARND_TRUE_MEAN}
\sup_{\theta_{0} \in l_{0}[p_{n}]} E_{\theta_{0}} \Pi \bigg(\theta:||\theta-\theta_0||^2 > M_{n}p_n\log\big(\frac{n}{p_n}\big)|X\bigg) \rightarrow 0,
 \end{equation}
for every $M_n\rightarrow \infty$ as $n \rightarrow \infty$.
\end{thm}

\begin{proof}
 A straight forward application of Markov's inequality coupled with the results of Theorem \ref{THM_3.1} and Theorem \ref{THM_3.2} leads to (\ref{OPT_POST_CON_ARND_TRUE_MEAN}).
\end{proof}

Note that (\ref{OPT_POST_CON_ARND_TRUE_MEAN}) implies that the posterior distributions contract around the true mean vector at a rate that can be at most of the order of the minimax rate in (\ref{MINIMAX_RATE}). On the other hand, it was shown in \citet{GGV2000} that the posterior distributions cannot contract faster than the minimax rate around the truth, that is, the rate of contraction cannot be of a smaller order compared to the minimax rate in (\ref{MINIMAX_RATE}). Hence, the rate of contraction around the true mean vector of the posterior distributions arising out of the general class of priors under consideration, with $a\in[\frac{1}{2},1)$, must be the of the order of the minimax optimal rate in (\ref{MINIMAX_RATE}), up to some multiplicative factors.\newline

It should, however, be noted that, for $\tau=\big(\frac{p_n}{n}\big)^{\alpha}$, with $\alpha\geqslant1$, the upper bound in (\ref{EXPECT_POST_VAR_UB}) results in a rate that is of a smaller order as compared to the minimax rate in (\ref{MINIMAX_RATE}). Consider, for example, the horseshoe or the standard double Pareto prior where one has $a=0.5$. Then taking $\tau=\frac{p_n}{n}$ in (\ref{EXPECT_POST_VAR_UB}), it follows that,
 \begin{equation}\label{SUBOPTIMAL_CONTRACTION}
\sup_{\theta_{0} \in l_{0}[p_{n}]} E_{\theta_{0}} \Pi \bigg(\theta:||\theta-T_{\tau}(X)||^2 > M_{n}p_n\sqrt{\log\big(\frac{n}{p_n}\big)}|X\bigg) \rightarrow 0,
 \end{equation}
for every $M_n\rightarrow \infty$ as $n \rightarrow \infty$. Equation (\ref{SUBOPTIMAL_CONTRACTION}) clearly shows that in such situations, the upper bound on the rate of contraction around the corresponding Bayes estimates, misses the minimax rate in (\ref{MINIMAX_RATE}) by the factor $\sqrt{\log(n/p_n)}$. Thus, for such choices of $\tau$, the corresponding posterior distributions contract too quickly around the corresponding Bayes estimates to be informative. However, if for each $a\in [0.5,1)$, we choose $\tau=\big(\frac{p_n}{n}\sqrt{\log(n/p_n)}\big)^{\frac{1}{2a}}$, the resulting posterior distributions contract around the corresponding Bayes estimators at least as fast as the minimax rate as given by the next theorem:

\begin{thm}\label{THM_3.4}
 Suppose $X \sim \mathcal{N}_{n}(\theta_{0},I_{n})$. If corresponding to each $a\in [0.5,1)$, we take $\tau=\big(\frac{p_n}{n}\sqrt{\log(n/p_n)}\big)^{\frac{1}{2a}}$, then the corresponding posterior distribution based on the general class of shrinkage priors (\ref{LAMBDA_PRIOR}), satisfies,
\begin{equation}\label{UB_CONTRACTION_RATE_ARND_BAYES_ESTIMATE}
\sup_{\theta_{0} \in l_{0}[p_{n}]} E_{\theta_{0}} \Pi \bigg(\theta:||\theta-T_{\tau}(X)||^2 > M_{n}p_n\log\big(\frac{n}{p_n}\big)|X\bigg) \rightarrow 0,
 \end{equation}
for every $M_n\rightarrow \infty$ as $n \rightarrow \infty$.
\end{thm}
\begin{proof}
The proof is almost immediate by an easy application of Markov's inequality together with the upper bound in (\ref{EXPECT_POST_VAR_UB}) by taking $\tau=\big(\frac{p_n}{n}\sqrt{\log(n/p_n)}\big)^{\frac{1}{2a}}$.
\end{proof}

Theorem \ref{THM_3.4} therefore gives an upper bound on the rate of contraction of the posterior distributions around the corresponding Bayes estimates and this bound equates the minimax rate up to some multiplicative factor. Note that choices of $\tau$, such as, $\tau=(\frac{p_n}{n}\sqrt{\log(n/p_n)})^{\frac{1}{2a}}$, depends on the value of the parameter $a$ in the definition $\pi(\lambda^2)$ in (\ref{LAMBDA_PRIOR}). Moreover, they also yield the minimax optimal rate which can be verified easily using the bounds given in Theorems \ref{THM_3.1} through \ref{THM_3.3}. This suggests that corresponding to each $a\in[0.5,1)$, a good choice of $\tau$ should be $\tau=(\frac{p_n}{n}\sqrt{\log(n/p_n)})^{\frac{1}{2a}}$ which should also depend on the kind of prior distributions that are in use. It should, however, be remembered that such choices of the global shrinkage parameter $\tau$ lack any physical interpretation as compared to the proportion of non-zero means $\frac{p_n}{n}$. However, from both theoretical and practical view point, one would perhaps be more interested in using priors with $a=0.5$, such as, the horseshoe or the standard double Pareto prior. In such situations, like \cite{PKV2014}, we also recommend using $\tau=\frac{p_n}{n}\sqrt{\log(n/p_n)}$ as a default option.\newline


To obtain a better insight about the spread of the posterior distribution around these estimators and the effect of choosing different values of $\tau$, let us confine our attention to the case when the function $L(\cdot)$ in (\ref{LAMBDA_PRIOR}) is non-decreasing over $(0,\infty)$ with $a=0.5$. The corresponding subclass includes the generalized double Pareto priors with shape parameter $\alpha=1$, the three parameter beta normal mixtures with parameters $a=0.5$ and $b>0$, the inverse gamma prior with shape parameter $\alpha=0.5$ and many more. See Section 2 of \citet{GTGC2015} in this context. The next theorem gives a lower bound to the total posterior variance corresponding to this sub-family of priors and that provides more insight into the effect of choosing $\tau$ depending on $\frac{p_n}{n}$. 

\begin{thm}\label{THM_3.5}
 Suppose $X \sim \mathcal{N}_{n}(\theta_{0},I_{n})$ and $\theta_{0} \in l_{0}[p_{n}]$. Further assume that the function $L(\cdot)$ given by (\ref{LAMBDA_PRIOR}) satisfies Assumption \ref{ASSUMPTION_LAMBDA_PRIOR} and is non-decreasing over $(0,\infty)$. Then for $a=\frac{1}{2}$, the variance of the posterior distribution corresponding to the general class of shrinkage priors, satisfies
 \begin{equation}\label{LOWER_BND_POST_VAR}
  \sum_{i=1}^{n} E_{\theta_{0}} Var(\theta_{0i}|X_i) \gtrsim (n-p_n)\tau\sqrt{\log\big(\frac{1}{\tau}\big)},
\end{equation}
 if $\tau \rightarrow 0$, as $n\rightarrow\infty$, $p_n\rightarrow \infty$ and $p_n=o(n)$.
\end{thm}
\begin{proof}
 See Appendix.
\end{proof}

Observe that the bounds in Theorems \ref{THM_3.2} and \ref{THM_3.5} coincide with each other both when $\tau=\frac{p_n}{n}\sqrt{\log(n/p_n)}$ and $\tau=\frac{p_n}{n}$. However, $\tau=\frac{p_n}{n}$ yields a rate that misses the minimax risk in (\ref{MINIMAX_RATE}) by the factor $\sqrt{\log(n/p_n)}$. On the other hand, the lower bound in (\ref{LOWER_BND_POST_VAR}) is of the order of the minimax risk in (\ref{MINIMAX_RATE}) for $\tau=\frac{p_n}{n}\sqrt{\log(n/p_n)}$. This again indicates that for this sub-family of priors, the corresponding posterior distributions contract around both the true mean vector and the corresponding Bayes estimates when $\tau=\frac{p_n}{n}\sqrt{\log(n/p_n)}$.

\section{Discussion}
We studied in this paper various theoretical properties of a general class of heavy-tailed continuous shrinkage priors in terms of the quadratic minimax risk for estimating a multivariate normal mean vector which is known to be sparse in the sense of being nearly black. It is shown that Bayes estimators arising out of this general class asymptotically attain the minimax risk in the $l_2$ norm possibly up to some multiplicative constants. Optimal rate of posterior contraction of these prior distributions in terms of the corresponding quadratic minimax rate has also been established. We provided a unifying theoretical treatment that holds for a very broad class of shrinkage priors including some well-known prior distributions such as the horseshoe prior, the normal-exponential-gamma priors, the three parameter beta normal priors, the generalized double Pareto priors, the inverse gamma priors and many others. Another major contribution of this work is to show that shrinkage priors which are appropriately heavy-tailed (already defined in Section 2) are good enough in order to attain the minimax optimal rate of contraction, provided that the global tuning parameter is carefully chosen, a question that was posed in \citet{PKV2014}. However, as commented in \citet{PKV2014}, for a full Bayes treatment by using a hyperprior for the global tuning parameter $\tau$, a pole at the origin may be required. As we already mentioned in Section 3 that one possible reason for such good performance of the kind of one-group shrinkage priors studied in this paper, is their ability to shrink the noise observations back to the origin, while leaving the large signals mostly unshrunk. Moreover, choice of the hyperparameter $a$ also plays a significant role for optimal posterior contraction of these priors. It should be noted that the range $[\frac{1}{2},1)$ of the hyperparameter $a$ is in concordance with that obtained in the context of multiple testing considered in \citet{GTGC2015}. The present work together with \citet{GTGC2015}, suggests that $a=0.5$ should be a better choice over values of $a\in(0.5,1)$ in the definition of $\pi(\lambda_i^2)$ in (\ref{LAMBDA_PRIOR}). We believe that the theoretical results in this paper can be extended further for a more general class of one-group priors through careful exploitation of properties of general slowly varying functions like those considered in \citet{GTGC2015}. We observed that (though not reported in this paper) when the number of non-zero means is unknown, the Bayes estimators based on this general class of one-group priors, combined with the empirical Bayes estimate of the global shrinkage parameter as suggested in \citet{PKV2014}, still attain the minimax risk up to a multiplicative constant in the $l_2$ norm. This follows quite easily using Lemma \ref{LEMMA_MOST_IMP} and the arguments used in the proof of Theorem 4.1 and Lemma A.7 in \citet{PKV2014}. In this sense, the present work can be considered as an extension of the posterior concentration properties for the horseshoe prior obtained by \citet{PKV2014} over a very large class of global-local scale mixture of normals. However, a more interesting and natural question is whether the kind of one-group priors studied in this work retain such optimal contraction properties if a hyperprior is assigned to the global shrinkage parameter $\tau$. This requires altogether different kinds of techniques and arguments as opposed to those used in this paper. We hope to address this problem elsewhere in future.

%
\appendix
\section*{Appendix}
\label{app}

\subsection{Proofs}
\begin{lem}\label{LEM_MOMENT_INEQ}
For the general class of shrinkage priors (\ref{LAMBDA_PRIOR}) satisfying Assumption \ref{ASSUMPTION_LAMBDA_PRIOR} the following holds true for any $0 <  a < 1$:
  \begin{equation}\label{KAPPA_EXP}
 E(1-\kappa \big| x,\tau) \leqslant \frac{KM}{a(1-a)} e^{\frac{x^2}{2}}\tau^{2a}(1+o(1)), \mbox{ each fixed $x \in \mathbb{R}$,} \nonumber
  \end{equation}
 where $\kappa=\frac{1}{1+\lambda^2\tau^2}$ denote the shrinkage coefficients and the $o(1)$ term depends only on $\tau^2$ such that $\lim_{\tau \rightarrow 0}o(1)=0$.
\end{lem}
\begin{proof}
 See \citet{GTGC2015}.
\end{proof}

\begin{lem}\label{LEM_CONCENTRATION_INEQ}
For every fixed $\tau > 0$, and each fixed $\eta, \delta \in (0,1),$ the posterior distribution of the shrinkage coefficients $\kappa=1/(1+\lambda^2\tau^2)$ based on the general class of shrinkage priors (\ref{LAMBDA_PRIOR}) satisfying Assumption \ref{ASSUMPTION_LAMBDA_PRIOR}, with $a > 0$, satisfies the following concentration inequality:
  \begin{eqnarray}
   \Pr(\kappa > \eta|x,\tau) &\leqslant& \frac{H(a,\eta,\delta)e^{-\frac{\eta(1-\delta)x^{2}}{2}}}{\tau^{2a}\Delta(\tau^2,\eta,\delta)}, \mbox{ uniformly in } x \in \mathbb{R},\nonumber\\
   \nonumber\\
  \mbox{where }  \Delta(\tau^2,\eta,\delta)&=&\xi(\tau^2,\eta,\delta)L\big(\frac{1}{\tau^2}(\frac{1}{\eta\delta}-1)\big),\nonumber\\ \nonumber\\
    \xi(\tau^2,\eta,\delta) &=& \frac{\int_{\frac{1}{\tau^2}\big(\frac{1}{\eta\delta}-1\big)}^{\infty}t^{-(a+\frac{1}{2}+1)}L(t)dt}{(a+\frac{1}{2})^{-1} \big(\frac{1}{\tau^2}\big(\frac{1}{\eta\delta}-1\big)\big)^{-(a+\frac{1}{2})}L(\frac{1}{\tau^2}\big(\frac{1}{\eta\delta}-1\big))},\quad  \mbox{and}\nonumber\\
    \nonumber\\
     H(a,\eta,\delta) &=& \frac{(a+\frac{1}{2}) (1-\eta\delta)^a}{ K(\eta\delta)^{(a+\frac{1}{2})}},\nonumber
  \end{eqnarray}
 where the term $\Delta(\tau^2,\eta,\delta)$ is such that $\lim_{\tau \rightarrow 0}\Delta(\tau^2,\eta,\delta)$ is a finite positive quantity for every fixed $\eta \in (0,1)$ and every fixed $\delta \in (0,1)$.
\end{lem}
\begin{proof}
 See \citet{GTGC2015}.
\end{proof}


\begin{lem}\label{LEMMA_MOST_IMP}
Let us consider the general class of shrinkage priors (\ref{LAMBDA_PRIOR}) satisfying Assumption \ref{ASSUMPTION_LAMBDA_PRIOR}, with $a>0$. Then, for $0 < \tau^2 < 1$ and given any $c > 2$, the absolute difference between the Bayes estimators $T_{\tau}(x) $ based on the aforesaid class of shrinkage priors and an observation $x$, can be bounded above by a real valued function $h(\cdot,\tau)$, depending on $c$, and satisfying the following:\\
 
For any $\rho > c$,
\begin{eqnarray}
 \lim_{\tau \downarrow 0}\sup_{\mid x \mid > \sqrt{\rho \log\big(\frac{1}{\tau^{2a}}\big)}} h(x,\tau) = 0. \nonumber
 \end{eqnarray}
\end{lem}
\begin{proof}
By definition, 
\begin{eqnarray}
\mid T_{\tau}(x) - x\mid
&=& \mid \frac{x\int_{0}^{1} \kappa\cdot \kappa^{a+\frac{1}{2}-1}(1-\kappa)^{-a-1}L(\frac{1}{\tau^2}(\frac{1}{\kappa}-1))e^{-\kappa x^2/2}d\kappa}{\int_{0}^{1} \kappa^{a+\frac{1}{2}-1}(1-\kappa)^{-a-1}L(\frac{1}{\tau^2}(\frac{1}{\kappa}-1))e^{-\kappa x^2/2}d\kappa}\mid\nonumber\\
&=& I(x,\tau), \mbox{ say.}\nonumber
\end{eqnarray}
Fix $\eta \in (0,1) $ and $\delta \in (0,1).$\\

Observe that
\begin{equation}\label{IMP_DECOMPSOSITION}
 I(x,\tau) \leqslant I_{1}(x,\tau) + I_{2}(x,\tau)
\end{equation}
where $I_{1}(x,\tau) = \mid xE(\kappa 1\{\kappa < \eta \} \mid x,\tau^2) \mid \nonumber $ and $  I_{2}(x,\tau) = \mid xE(\kappa 1\{\kappa > \eta \} \mid x,\tau^2) \mid \nonumber. $\\

Now using the variable transformation $t=\frac{1}{\tau^2}(\frac{1}{\kappa}-1),$ we have the following:
\begin{eqnarray}
  I_{1}(x,\tau)
 &=& \mid \frac{x\int_{0}^{\eta} \kappa\cdot \kappa^{a+\frac{1}{2}-1}(1-\kappa)^{-a-1}L(\frac{1}{\tau^2}(\frac{1}{\kappa}-1))e^{-\kappa x^2/2}d\kappa}{\int_{0}^{1} \kappa^{a+\frac{1}{2}-1}(1-\kappa)^{-a-1}L(\frac{1}{\tau^2}(\frac{1}{\kappa}-1))e^{-\kappa x^2/2}d\kappa}\mid\nonumber\\
 &=& \mid \frac{x\int_{\frac{1}{\tau^2}(\frac{1}{\eta}-1)}^{\infty} \frac{1}{(1+t\tau^2)^{3/2}}t^{-a-1}L(t)e^{-\frac{x^2}{2(1+t\tau^2)}}dt}{\int_{0}^{\infty}\frac{1}{(1+t\tau^2)^{1/2}}t^{-a-1}L(t)e^{-\frac{x^2}{2(1+t\tau^2)}}dt}\mid\nonumber \\
 &\leqslant& \mid \frac{x\int_{\frac{1}{\tau^2}(\frac{1}{\eta}-1)}^{\infty}\frac{1}{(1+t\tau^2)^{3/2}}t^{-a-1}L(t)e^{-\frac{x^2}{2(1+t\tau^2)}}dt}{\int_{\frac{t_{0}}{\tau^2}}^{\infty}\frac{1}{(1+t\tau^2)^{1/2}}t^{-a-1}L(t)e^{-\frac{x^2}{2(1+t\tau^2)}}dt}\mid \mbox{}\nonumber\\
 &=& J_{1}(x,\tau)\mbox{ say,} \label{IMP_DECOMPSOSITION_1_1}
\end{eqnarray}

Next observe that $\frac{t_{0}}{\tau^2} > t_{0}$ as $\tau^2 < 1$. Hence $L(t) \geqslant c_{0}$ for every $t \geqslant \frac{t_{0}}{\tau^2}$. Also, the function $L$ is bounded by the constant $M >0$. Utilizing these two observations and using the variable transformation $u=\frac{x^2}{1+t\tau^2}$ in both the numerator and the denominator of $J_{1}(x,\tau)$ in (\ref{IMP_DECOMPSOSITION_1_1}), and writing $s=\frac{1}{1+t_{0}} \in (0,1)$, we see that the term $J_{1}(x,\tau)$ can be bounded above as follows:
\begin{eqnarray}
J_{1}(x,\tau)
&\leqslant& \frac{M}{c_{0}} \mid x \frac{\int_{0}^{\eta x^2}e^{-u/2}\big(\frac{u}{x^2}\big)^{3/2}\big(\frac{1}{\tau^2}\big(\frac{x^2}{u}-1\big)\big)^{-a-1} \frac{x^2}{\tau^2u^2}du}{\int_{0}^{sx^2}e^{-u/2}\big(\frac{u}{x^2}\big)^{1/2}\big(\frac{1}{\tau^2}\big(\frac{x^2}{u}-1\big)\big)^{-a-1}\frac{x^2}{\tau^2u^2}du} \mid \nonumber\\
&=& \frac{M}{c_{0}} \mid \frac{1}{x} \cdot \frac{\int_{0}^{\eta x^2}e^{-u/2}u^{a+3/2-1}\big(1-\frac{u}{x^2}\big)^{-a-1}du}{\int_{0}^{s x^2}e^{-u/2}u^{a+1/2-1}\big(1-\frac{u}{x^2}\big)^{-a-1}du} \mid\nonumber
\end{eqnarray}
Note that $0 < u < \eta x^2 \Rightarrow 0 < \frac{u}{x^2} < \eta < 1 \Rightarrow 1-\eta < 1 -\frac{u}{x^2} < 1.$ Similarly, $0 < u < s x^2 \Rightarrow 1-s < 1 -\frac{u}{x^2} < 1.$ Therefore we have,
\begin{eqnarray}
J_{1}(x,\tau)
 &\leqslant& \frac{M}{c_{0}(1-\eta)^{1+a}} \mid \frac{1}{x} \cdot \frac{\int_{0}^{\eta x^2}e^{-u/2}u^{a+3/2-1}du}{\int_{0}^{s x^2}e^{-u/2}u^{a+1/2-1}du} \mid  \nonumber\\
 &\leqslant& \frac{M}{c_{0}(1-\eta)^{1+a}} \mid \frac{1}{x} \cdot \frac{\int_{0}^{\infty}e^{-u/2}u^{a+3/2-1}du}{\int_{0}^{s x^2}e^{-u/2}u^{a+1/2-1}du} \mid  \nonumber\\
 &=& h_{1}(x,\tau)\mbox{ say},\label{IMP_DECOMPSOSITION_1_2}
\end{eqnarray}
where $h_{1}(x,\tau)=C_{\textasteriskcentered} \big[\mid x \int_{0}^{s x^2}e^{-u/2}u^{a+1/2-1}du \mid\big]^{-1}  $ for some $C_{\textasteriskcentered} \equiv C_{\textasteriskcentered}(a,\eta,L) > 0$ which is independent of both $x$ and $\tau$. Note that the function $h_{1}(x,\tau)$ is actually independent of $\tau$ and depends on $x$ only.\\

Next we observe that,
\begin{eqnarray}
I_{2}(x,\tau)
&=&\mid xE(\kappa 1\{\kappa > \eta \} \mid x,\tau^2) \mid \nonumber\\
&\leqslant& \mid x\Pr(\kappa > \eta\mid x,\tau^2) \mid\nonumber\\
&\leqslant& \mid x\frac{H(a,\eta,\delta)e^{-\frac{\eta(1-\delta) x^2}{2}}}{\tau^{2a}\Delta(\tau^2,\eta,\delta)} \mid \nonumber\\
&=& h_{2}(x,\tau)\mbox{ say},\label{IMP_DECOMPSOSITION_2}
\end{eqnarray}
 
Let $h(x,\tau) = h_{1}(x,\tau)+h_{2}(x,\tau)$. Therefore combining (\ref{IMP_DECOMPSOSITION}), (\ref{IMP_DECOMPSOSITION_1_1}), (\ref{IMP_DECOMPSOSITION_1_2}) and (\ref{IMP_DECOMPSOSITION_2}), we finally obtain for every $x \in \mathbb{R} $ and $\tau > 0$,
\begin{equation}
 \mid T_{\tau}(x)-x \mid \leqslant h(x,\tau).\label{UB_FUNCTION}
\end{equation}

Now observe that the function $h_{1}(x,\tau)$ is strictly decreasing in $\mid x \mid$. Therefore, for any fixed $\tau >0$ and every $\rho > 0$,
\begin{eqnarray}
 \sup_{\mid x \mid > \sqrt{\rho \log\big(\frac{1}{\tau^{2a}}\big)}} h_{1}(x,\tau)
 &\leqslant& C_{\textasteriskcentered} \big[\mid \sqrt{\rho \log\big(\frac{1}{\tau^{2a}}\big)} \int_{0}^{s\rho \log\big(\frac{1}{\tau^{2a}}\big)}e^{-u/2}u^{a+1/2-1}du \mid\big]^{-1} \nonumber
\end{eqnarray}
implying that 
\begin{equation}\label{LIMSUP_1}
 \lim_{\tau \downarrow 0}\sup_{\mid x \mid > \sqrt{\rho \log\big(\frac{1}{\tau^{2a}}\big)}} h_{1}(x,\tau) = 0.
\end{equation}

Again the function $h_{2}(x,\tau)$ is eventually decreasing in $|x|$. Therefore, for all sufficiently small $\tau > 0$,
\begin{equation}
 \sup_{\mid x \mid > \sqrt{\rho \log\big(\frac{1}{\tau^{2a}}\big)}} h_{2}(x,\tau) \leqslant h_{2}(\sqrt{\rho \log\big(\frac{1}{\tau^{2a}}\big)},\tau).\nonumber
\end{equation}

Let $\beta\equiv \beta(\eta,\delta) = \lim_{\tau \rightarrow 0} \Delta(\tau^2,\eta,\delta)$ for every fixed $\eta ,\delta \in (0,1)$. Then $ 0 < \beta < \infty$ which follows from Lemma \ref{LEM_CONCENTRATION_INEQ}. Then, 
\begin{eqnarray}
\lim_{\tau \rightarrow 0} h_{2}(\sqrt{\rho \log\big(\frac{1}{\tau^{2a}}\big)},\tau)
&=& \frac{1}{\beta} \lim_{\tau \rightarrow 0} \mid \tau^{-2a}\sqrt{\rho\log\big(\frac{1}{\tau^{2a}}\big)}e^{- \frac{\eta(1-\delta)}{2}\rho\log\big(\frac{1}{\tau^{2a}}\big)}\mid \nonumber\\
&=& \frac{\sqrt{\rho}}{\alpha} \lim_{\tau \rightarrow 0} \big(\tau^{2a})^{\frac{\eta(1-\delta)}{2}\big(\rho-\frac{2}{\eta(1-\delta)}\big)} \sqrt{\log\big(\frac{1}{\tau^{2a}}\big)}\nonumber\\
&=& \left\{ \begin{array}{rl}
0 &\mbox{ if $\rho > \frac{2}{\eta(1-\delta)} $} \\
\infty &\mbox{ otherwise,}
\end{array}\right.\nonumber
\end{eqnarray}

whence it follows that
\begin{eqnarray}
\lim_{\tau \rightarrow 0} \sup_{\mid x \mid > \sqrt{\rho \log\big(\frac{1}{\tau^{2a}}\big)}} h_{2}(x,\tau)
&=& \left\{ \begin{array}{rl}
0 &\mbox{ if $\rho > \frac{2}{\eta(1-\delta)} $} \\
\infty &\mbox{ otherwise,}
\end{array}\right.\label{LIMSUP_2}
\end{eqnarray}

Combining (\ref{LIMSUP_1}) and (\ref{LIMSUP_2}) together with the fact that 
$$\lim_{\tau \rightarrow 0} \sup_{\mid x \mid > \sqrt{\rho \log\big(\frac{1}{\tau^{2a}}\big)}} h(x,\tau) \leqslant \lim_{\tau \rightarrow 0} \sup_{\mid x \mid > \sqrt{\rho \log\big(\frac{1}{\tau^{2a}}\big)}} h_{1}(x,\tau) + \lim_{\tau \rightarrow 0} \sup_{\mid x \mid > \sqrt{\rho \log\big(\frac{1}{\tau^{2a}}\big)}} h_{2}(x,\tau) $$
it follows that
\begin{eqnarray}
\lim_{\tau \rightarrow 0} \sup_{\mid x \mid > \sqrt{\rho \log\big(\frac{1}{\tau^{2a}}\big)}} h(x,\tau)
&=& \left\{ \begin{array}{rl}
0 &\mbox{ if $\rho > \frac{2}{\eta(1-\delta)} $} \\
\infty &\mbox{ otherwise,}
\end{array}\right.\label{LIMSUP_FINAL}
\end{eqnarray}

Observe that by choosing $\eta$ appropriately close to 1 and $\delta$ close to 0, any real number larger than 2 can be expressed in the form $\frac{2}{\eta(1-\delta)}.$ For example, taking $\eta=\frac{5}{6}$ and $\delta=\frac{1}{5}$ we obtain $\frac{2}{\eta(1-\delta)}=3$. Hence, given $c > 2$, let us choose $0 < \eta, \delta < 1 $ such that $c=\frac{2}{\eta(1-\delta)}$. Clearly, the choice of $h(\cdot,\tau)$ depends on $c$. The preceding discussion, coupled with (\ref{UB_FUNCTION}) and (\ref{LIMSUP_FINAL}), completes the proof of Lemma \ref{LEMMA_MOST_IMP}.
\end{proof}

\paragraph{Remark A.1}
Observe that the function $h(\cdot,\tau)$ defined in the proof of Lemma \ref{LEMMA_MOST_IMP} satisfies the following:

\begin{equation}
 \lim_{|x| \rightarrow \infty}h(x,\tau)=0, \mbox{ for each fixed $\tau > 0$}.\nonumber
\end{equation}

This means (using Lemma \ref{LEMMA_MOST_IMP}) that for any fixed $\tau > 0$, we have,
\begin{equation}
 \lim_{|x| \rightarrow \infty}\mid T_{\tau}(x)-x \mid =0. \label{SIGNALS_UNSHRUNK}
\end{equation}

Equation (\ref{SIGNALS_UNSHRUNK}) above shows that for the general class of tail robust priors under consideration, large observations almost remain unshrunk no matter however small $\tau$ is.

\paragraph{Remark A.2}
It should be noted that the choice of the function $h$ as in Lemma \ref{LEMMA_MOST_IMP} is not unique since it depends on $c >1$. Moreover, some other function, say, $\widetilde{h}$, can easily be obtained through exploiting the integrals in the proof of Lemma \ref{LEMMA_MOST_IMP} in a different manner.

\paragraph{Proof of Theorem \ref{THM_3.1}}
\begin{proof}
Suppose that $X \sim \mathcal{N}_{n}(\theta,I_{n})$, $\theta \in l_{0}[p_{n}]$ and $\tilde{p_n}=\#\{i:\theta_i\neq 0\}$. Let us split the term $E_{\theta}||T_{\tau}(X)-\theta||^2=\sum\limits_{i=1}^{n} E_{\theta_{i}} \big(T_{\tau}(X_{i})-\theta_{i}\big)^2 $ into two parts as follows:
\begin{equation}\label{THM_3.1_MAIN_DECOMPOSITION}
\sum_{i=1}^{n} E_{\theta_{i}} \big(T_{\tau}(X_{i})-\theta_{i}\big)^2
= \sum_{i:\theta_i\neq 0} E_{\theta_{i}} \big(T_{\tau}(X_{i})-\theta_{i}\big)^2 + \sum_{i:\theta_i=0} E_{\theta_{i}} \big(T_{\tau}(X_{i})-\theta_{i}\big)^2
\end{equation}
We shall show now that the terms on the right hand side of (\ref{THM_3.1_MAIN_DECOMPOSITION}) can be bounded above by $\tilde{p_n}\log\big(\frac{1}{\tau^{2a}}\big)$ and $(n-\tilde{p_n})\tau^{2a}\sqrt{\log\big(\frac{1}{\tau^{2a}}\big)}$, respectively, up to some multiplicative constants, for all sufficiently small $\tau < 1$.\\

{\it Non-zero means:}\\

Let $\zeta_{\tau}=\sqrt{2\log\big(\frac{1}{\tau^{2a}}\big)}.$\\

Then,
\begin{eqnarray}
 E_{\theta_{i}}\big(T_{\tau}(X_i) - \theta_{i} \big)^2
 &=& E_{\theta_{i}}\big( \big(T_{\tau}(X_i) - X_i \big) + \big(X_i - \theta_{i}\big) \big)^2\nonumber\\
 &\leqslant& E_{\theta_{i}}\big(T_{\tau}(X_i) - X_i\big)^2+1+2 \sqrt{E_{\theta_{i}}\big(T_{\tau}(X_i) - X_i\big)^2} \nonumber\\
 &=& \bigg[\sqrt{E_{\theta_{i}}\big(T_{\tau}(X_i) - X_i\big)^2}+1\bigg]^2 \label{THM_3.1_UB1}
\end{eqnarray}
where the step preceding the final step in the above chain of inequalities follows from H\"{o}lder's inequality coupled with the fact $E_{\theta_{i}}\big(X_i - \theta_{i}\big)^2=1$.\\

%
Let us now fix any $c>1$ and choose any $\rho>c$. Then, using Lemma \ref{LEMMA_MOST_IMP}, there exists a non-negative real-valued function $h(\cdot,\tau)$, depending on $c$, such that
 \begin{equation}
  |T_{\tau}(x)-x| \leqslant h(x,\tau), \mbox{ for all } x \in \mathbb{R}\label{THM_3.1_LEM_4.3_UB1}
\end{equation}
and
 \begin{equation}
  \lim_{\tau \downarrow 0}\sup_{\mid x \mid > \rho \zeta_{\tau}} h(x,\tau) = 0.\label{THM_3.1_LEM_4.3_UB2}
\end{equation}

Once again, using the fact $\big(T_{\tau}(X_i) - X_i\big)^2 \leqslant X_i^2$, together with (\ref{THM_3.1_LEM_4.3_UB1}), we obtain,
\begin{eqnarray}
 E_{\theta_{i}}\big(T_{\tau}(X_i) - X_i\big)^2
 &=& E_{\theta_{i}}\big[(T_{\tau}(X_i) - X_i\big)^2 1\big\{|X_i|\leqslant \rho\zeta_{\tau}\big\}\big]\nonumber\\
 && + \mbox{ } E_{\theta_{i}}\big[(T_{\tau}(X_i) - X_i\big)^2 1\big\{|X_i|> \rho\zeta_{\tau}\big\}\big]\nonumber\\
 &\leqslant& \rho^2\zeta_{\tau}^2+\bigg(\sup_{|x|> \rho\zeta_{\tau}}h(x,\tau)\bigg)^2 \label{THM_3.1_UB2}
\end{eqnarray}
Now using (\ref{THM_3.1_LEM_4.3_UB2}) and the fact that $\zeta_{\tau}\rightarrow\infty$ as $\tau\rightarrow 0$, it follows that,
\begin{equation}
 \bigg(\sup_{|x|> \rho\zeta_{\tau}}h(x,\tau)\bigg)^2=o(\zeta_{\tau}^2).\label{THM_3.1_h_eq_o_zeta}
\end{equation}
On combining (\ref{THM_3.1_UB2}) and (\ref{THM_3.1_h_eq_o_zeta}), it follows that,
\begin{equation}
 E_{\theta_{i}}\big(T_{\tau}(X_i) - X_i\big)^2 \leqslant \rho^2\zeta_{\tau}^2(1+o(1)), \quad\mbox{as }\tau \rightarrow 0. \label{THM_3.1_UB_EXPCT_DIFF_SQ}
\end{equation}
Note that (\ref{THM_3.1_UB_EXPCT_DIFF_SQ}) holds uniformly for any $i$ such that $\theta_i \neq 0$, whence we have,
\begin{equation}
 \sum_{i:\theta_i\neq 0} E_{\theta_{i}} \big(T_{\tau}(X_{i})-\theta_{i}\big)^2 \lesssim \tilde{p_n}\zeta_{\tau}^2, \quad \mbox{as }\tau \rightarrow 0.\label{THM_3.1_UB_NONZERO_THETA} 
\end{equation}
 
{\it Zero means:}\\

 We split up the term for the zero means as follows:
 \begin{equation}
 E_{0}\big[T_{\tau}(X)^{2}\big]=E_{0}\big[T_{\tau}(X)^{2}1\{|X|\leqslant \zeta_{\tau}\}\big]+E_{0}\big[T_{\tau}(X)^{2}1\{|X| >\zeta_{\tau}\}\big],\label{THM_3.1_EQ_PETZ_1}
 \end{equation}
 where $\zeta_{\tau}=\sqrt{2\log\big(\frac{1}{\tau^{2a}}\big)}.$\\
 
 For the first term on the right hand side of (\ref{THM_3.1_EQ_PETZ_1}), denoting $g_1(a)=\frac{KM}{a(1-a)}$ and using Lemma \ref{LEM_MOMENT_INEQ}, we obtain,
 \begin{eqnarray}
  E_{0} T_{\tau}(X)^{2} 1\{|X|\leqslant \zeta_{\tau}\}
  &\leqslant& \frac{g_1(a)^2}{\sqrt{2\pi}} (\tau^{2a})^2 \int_{-\zeta_{\tau}}^{\zeta_{\tau}} x^2 e^{\frac{x^2}{2}}dx (1+o(1)) \mbox{ as } \tau \rightarrow 0 \nonumber\\
  &=& \frac{2g_1(a)^2}{\sqrt{2\pi}} (\tau^{2a})^2 \int_{0}^{\zeta_{\tau}} x^2 e^{\frac{x^2}{2}}dx (1+o(1)) \mbox{ as } \tau \rightarrow 0 \nonumber\\
  &\leqslant& \sqrt{\frac{2}{\pi}}g_1(a)^2(\tau^{2a})^2 \zeta_{\tau}\frac{1}{\tau^{2a}}(1+o(1)) \mbox{ as } \tau \rightarrow 0 \nonumber\\
  &\lesssim& \zeta_{\tau}\tau^{2a}\label{THM_3.1_EQ_PETZ_2}
 \end{eqnarray}
 
 while, for the second term using the fact that $|T_{\tau}(x)|\leqslant |x|$ for $x\in\mathbb{R}$ we have,
 \begin{eqnarray}
 E_{0} T_{\tau}(X)^{2} 1\{|X| > \zeta_{\tau}\}
 &\leqslant& 2 \int_{\zeta_{\tau}}^{\infty} x^2\phi(x)dx\nonumber\\
 &=& 2 \big[\zeta_{\tau}\phi(\zeta_{\tau})+(1-\Phi(\zeta_{\tau})) \big]\nonumber\\
 &\leqslant& 2\zeta_{\tau}\phi(\zeta_{\tau})+2\frac{\phi(\zeta_{\tau})}{\zeta_{\tau}}\nonumber\\
 &=& \sqrt{\frac{2}{\pi}}\zeta_{\tau}\tau^{2a}(1+o(1))\mbox{ as } \tau \rightarrow 0 \nonumber\\
 &\lesssim& \zeta_{\tau}\tau^{2a} \label{THM_3.1_EQ_PETZ_3}
 \end{eqnarray}
Combining equations (\ref{THM_3.1_EQ_PETZ_1}), (\ref{THM_3.1_EQ_PETZ_2}) and (\ref{THM_3.1_EQ_PETZ_3}), it follows that,
\begin{eqnarray}
\sum_{i:\theta_i=0}E_{\theta_{i}}\big(T_{\tau}(X_i)-\theta_{i} \big)^2
&\lesssim& (n-\tilde{p_n})\tau^{2a}\sqrt{\log\big(\frac{1}{\tau^{2a}}\big)} \quad\mbox{as }\tau \rightarrow 0 \label{THM_3.1_UB_ZERO_THETA}
\end{eqnarray}
Finally, on combining (\ref{THM_3.1_MAIN_DECOMPOSITION}), (\ref{THM_3.1_UB_NONZERO_THETA}) and (\ref{THM_3.1_UB_ZERO_THETA}), we have,
\begin{equation}
 \sum_{i=1}^{n} E_{\theta_{i}} \big(T_{\tau}(X_{i})-\theta_{i}\big)^2
 \lesssim \tilde{p_n}\log\big(\frac{1}{\tau^{2a}}\big)+ (n-\tilde{p_n})\tau^{2a}\sqrt{\log\big(\frac{1}{\tau^{2a}}\big)}
\end{equation}
The stated result now becomes immediate by observing that $\tilde{p_n}\leqslant p_n$. This completes the proof of Theorem \ref{THM_3.1}.

\end{proof}

\begin{lem}\label{POST_VAR_IDENTITY}
The posterior variance arising out of the general class of shrinkage priors (\ref{LAMBDA_PRIOR}) can be represented by the following identity:
 \begin{eqnarray}
  Var(\theta|x)
  &=&\frac{T_{\tau}(x)}{x} - \big(T_{\tau}(x) -x \big)^2 + x^2 \frac{\int_{0}^{\infty}\frac{1}{(1+t\tau^2)^{5/2}}t^{-a-1}L(t)e^{-\frac{x^2}{2(1+t\tau^2)}}dt}{\int_{0}^{\infty}\frac{1}{(1+t\tau^2)^{1/2}}t^{-a-1}L(t)e^{-\frac{x^2}{2(1+t\tau^2)}}dt}\label{VAR_IDNT_1}\\
   &=& \frac{T_{\tau}(x)}{x} - T_{\tau}^2(x) + x^2 \frac{\int_{0}^{\infty}\frac{(t\tau^2)^2}{(1+t\tau^2)^{5/2}}t^{-a-1}L(t)e^{-\frac{x^2}{2(1+t\tau^2)}}dt}{\int_{0}^{\infty}\frac{1}{(1+t\tau^2)^{1/2}}t^{-a-1}L(t)e^{-\frac{x^2}{2(1+t\tau^2)}}dt}\label{VAR_IDNT_2} 
 \end{eqnarray}
 which can be bounded from above by
 \begin{equation}
  Var(\theta|x) \leqslant 1 + x^2.\nonumber
 \end{equation}
\end{lem}

\begin{proof} By the law of iterated variance it follows that
\begin{eqnarray}
 Var(\theta|x)
 &=& E\big[Var(\theta|x,\kappa,\tau)\big]+ Var\big[E(\theta|x,\kappa,\tau)\big]\nonumber\\
 &=& E\big[(1-\kappa)|x,\tau\big]+Var\big[x(1-\kappa)|x,\tau)\big]\nonumber\\
 &=& E\big[(1-\kappa)|x,\tau\big]+ x^2 Var\big[\kappa|x,\tau)\big]\nonumber\\
 &=& E\big[(1-\kappa)|x,\tau\big]+ x^2 E\big[\kappa^2|x,\tau\big] - x^2E^2\big[\kappa|x,\tau\big]\nonumber\\
 &=& \frac{T_{\tau}(x)}{x} - \big(T_{\tau}(x) -x \big)^2 + x^2 \frac{\int_{0}^{\infty}\frac{1}{(1+t\tau^2)^{5/2}}t^{-a-1}L(t)e^{-\frac{x^2}{2(1+t\tau^2)}}dt}{\int_{0}^{\infty}\frac{1}{(1+t\tau^2)^{1/2}}t^{-a-1}L(t)e^{-\frac{x^2}{2(1+t\tau^2)}}dt}\nonumber
\end{eqnarray}
 which can equivalently be represented as by the following identity as well:
 \begin{eqnarray}
 Var(\theta|x)
 &=& E\big[(1-\kappa)|x,\tau\big]+ x^2 E\big[(1-\kappa)^2|x,\tau\big] - x^2E^2\big[1-\kappa|x,\tau\big]\nonumber\\
 &=& \frac{T_{\tau}(x)}{x} - T_{\tau}^2(x) + x^2 \frac{\int_{0}^{\infty}\frac{(t\tau^2)^2}{(1+t\tau^2)^{5/2}}t^{-a-1}L(t)e^{-\frac{x^2}{2(1+t\tau^2)}}dt}{\int_{0}^{\infty}\frac{1}{(1+t\tau^2)^{1/2}}t^{-a-1}L(t)e^{-\frac{x^2}{2(1+t\tau^2)}}dt}.\nonumber
 \end{eqnarray}
 That $Var(\theta|x) \leqslant 1 + x^2$ now follows trivially from the above identities.
\end{proof}

\begin{lem}\label{LEMMA_IMPORTANT_INTEGRAL}
 Let us define
 \begin{equation}
  J(x,\tau)=x^2 \frac{\int_{0}^{\infty}\frac{(t\tau^2)^2}{(1+t\tau^2)^{5/2}}t^{-a-1}L(t)e^{-\frac{x^2}{2(1+t\tau^2)}}dt}{\int_{0}^{\infty}\frac{1}{(1+t\tau^2)^{1/2}}t^{-a-1}L(t)e^{-\frac{x^2}{2(1+t\tau^2)}}dt}, x\in\mathbb{R} \mbox{ and } \tau\in(0,1),\nonumber
 \end{equation}
 where the function $L(\cdot)$ is already defined in (\ref{LAMBDA_PRIOR}) and satisfies Assumption \ref{ASSUMPTION_LAMBDA_PRIOR}, with $a\in[0.5,1)$. Then, for each $x\in\mathbb{R}$ and every $0<\tau<1$, the function $J(\cdot,\cdot)$ is bounded above by,
 \begin{equation}
  J(x,\tau) \leqslant 2KMe^{\frac{x^2}{2}}\tau^{2a}\big(1+o(1)),\label{CRUCIAL_UB_VAR}
 \end{equation}
 where the $o(1)$ term is independent of $x$, and depends only on $\tau^2$ such that the term $(1+o(1))$ in (\ref{CRUCIAL_UB_VAR}) is positive for any $0<\tau<1$, and $\lim_{\tau\rightarrow 0}o(1)=0$.

%
\end{lem}

\begin{proof}
First observe that, for each fixed $\tau\in (0,1)$, the function $J(x,\tau)$ is symmetric in $x$, and so is the stated upper bound in (\ref{CRUCIAL_UB_VAR}). Moreover, $J(0,\tau)=0$, for any $0<\tau<1$. Thus, the stated result is vacuously true when $x=$ and $0<\tau<1$. Therefore, it will be enough to prove that (\ref{CRUCIAL_UB_VAR}) holds when $x>0$. So, let us assume that $x>0$.\\

Note that
 \begin{eqnarray}
J(x,\tau)
&\leqslant& Mx^2e^{\frac{x^2}{2}} \frac{\int_{0}^{\infty}\frac{(t\tau^2)^2}{(1+t\tau^2)^{5/2}}t^{-a-1}e^{-\frac{x^2}{2(1+t\tau^2)}}dt}{\int_{0}^{\infty}\frac{1}{(1+t\tau^2)^{1/2}}t^{-a-1}L(t)dt}\nonumber\\
&=& KMx^2e^{\frac{x^2}{2}} \int_{0}^{\infty}\frac{(t\tau^2)^2}{(1+t\tau^2)^{5/2}}t^{-a-1}e^{-\frac{x^2}{2(1+t\tau^2)}}dt \big(1+o(1)\big)\label{VAR_NONZERO_2}
\end{eqnarray}
where in the preceding chain of inequalities we use the facts that the function $L(\cdot)$ is bounded above by the constant $M > 0$ and $\int_{0}^{\infty}\frac{1}{(1+t\tau^2)^{1/2}}t^{-a-1}L(t)dt=\int_{0}^{\infty}t^{-a-1}L(t)dt\big(1+o(1)\big)= K^{-1}\big(1+o(1)\big)$ as $\tau \rightarrow 0$ (which follows from Lebesgue's Dominated Convergence Theorem). Clearly, the $o(1)$ term does not involve $x$ and depends only on $\tau^2$ such that $\lim_{\tau\rightarrow 0}o(1)=0$. It is also evident that the term $(1+o(1))$ in (\ref{CRUCIAL_UB_VAR}) is always positive for any $0<\tau<1$.\\

Consider now the following variable transformation in the integral in (\ref{VAR_NONZERO_2}):
\begin{equation}
 u=\frac{x^2}{1+t\tau^2}.\nonumber
\end{equation}

Then we have,
\begin{eqnarray}
J(x,\tau)
&\leqslant&  KMx^2e^{\frac{x^2}{2}} \int_{0}^{x^2} \bigg(1-\frac{u}{x^2}\bigg)^2\bigg(\frac{u}{x^2}\bigg)^{1/2}\bigg(\frac{1}{\tau^2}\frac{x^2}{u}\big(1-\frac{u}{x^2}\big)\bigg)^{-a-1}e^{-u/2}\frac{x^2}{\tau^2u^2}du\big(1+o(1)\big)\nonumber\\
&=&  KM(x^2)^{1/2-a}e^{\frac{x^2}{2}}\tau^{2a}\int_{0}^{x^2} \bigg(1-\frac{u}{x^2}\bigg)^{1-a}u^{a-1/2}e^{-u/2}du\big(1+o(1)\big)\label{VAR_NONZERO_3}
\end{eqnarray}

Note that $0 < u < x^2 \Longrightarrow 0 < 1-\frac{u}{x^2} < 1$. Hence, $\big(1-\frac{u}{x^2}\big)^{1-a}<1$ as $0<a<1$. Also, as $\frac{1}{2}\leqslant a <1$, we have, $u^{a-1/2}\leqslant (x^2)^{a-1/2}$.\\

Therefore using (\ref{VAR_NONZERO_3}), we obtain,
\begin{eqnarray}\label{VAR_NONZERO_4}
J(x,\tau)
&\leqslant&  KMe^{\frac{x^2}{2}}\tau^{2a}\int_{0}^{x^2} e^{-u/2}du\big(1+o(1)\big),\nonumber\\
&=& 2KMe^{\frac{x^2}{2}}\tau^{2a} \big(1-e^{-x^2/2}\big)\big(1+o(1)\big)\nonumber\\
&\leqslant& 2KMe^{\frac{x^2}{2}}\tau^{2a}\big(1+o(1)\big),\nonumber
\end{eqnarray}
thereby completing the proof of Lemma \ref{LEMMA_IMPORTANT_INTEGRAL}.
\end{proof}

\paragraph{Proof of Theorem \ref{THM_3.2}}
\begin{proof}
 Suppose that $X \sim \mathcal{N}_{n}(\theta,I_{n})$, $\theta \in l_{0}[p_{n}]$ and $\tilde{p_n}=\#\{i:\theta_{i}\neq 0\}$. Then $\tilde{p_n} \leqslant p_n$. Let $\zeta_{\tau}=\sqrt{2\log\big(\frac{1}{\tau^{2a}}\big)}$.\\
 
 {\it Nonzero means:}\newline

 By applying the same reasoning as in the proof of Lemma \ref{LEMMA_MOST_IMP} to the final term of $Var(\theta|x)$ in (\ref{VAR_IDNT_1}), there exists a non-negative real-valued function $\tilde{h}(x,\tau)$ such that $Var(\theta|x) \leqslant \tilde{h}(x,\tau)$, where $\tilde{h}(x,\tau) \rightarrow 1 $ as $x \rightarrow \infty$ for any fixed $\tau \in (0,1)$. If $\tau \rightarrow 0$, the function $\tilde{h}(x,\tau)$ satisfies the following for any $c > 1$:
 \begin{eqnarray}
  \lim_{\tau \downarrow 0}\sup_{\mid x \mid > \rho \zeta_{\tau}} \tilde{h}(x,\tau) = 1 \mbox{ for all } \rho > c. \nonumber
 \end{eqnarray}

Hence $Var(\theta|x) \lesssim 1,$ for any $|x| > \zeta_{\tau}$ as $\tau \rightarrow 0$.\\

Let us now consider the case $|x| \leqslant \zeta_{\tau}$. Using the identity (\ref{VAR_IDNT_2}) and then applying Lemma \ref{LEMMA_IMPORTANT_INTEGRAL}, we obtain:
\begin{equation}
 Var(\theta|x)
\leqslant 1+ 2KMe^{\frac{x^2}{2}}\tau^{2a}\big(1+o(1)\big),\nonumber
\end{equation}
where the $o(1)$ term above is independent of $x$, and depends only on $\tau^2$ such that the term $(1+o(1))$ is always positive for any $0<\tau<1$, and $\lim_{\tau\rightarrow 0}o(1)=0$.\\

Note that $e^{\frac{x^2}{2}}\tau^{2a}\leqslant 1$ when $|x| \leqslant \zeta_{\tau}$. Thus, $Var(\theta|x)\leqslant 1+2KM(1+o(1))$, for $|x| \leqslant \zeta_{\tau}$.\\

Therefore,
\begin{eqnarray}
 E_{\theta_{i}} Var(\theta_{i}|X_i)1\{|X_i|\leqslant \zeta_{\tau}\} \lesssim 1+2KM.\nonumber
\end{eqnarray}

Using the preceding discussion it therefore follows that for any $i$ such that $\theta_i\neq 0$,
\begin{eqnarray}
\begin{aligned}
&\quad E_{\theta_{i}} Var(\theta_{i}|X_i)\nonumber\\
& =E_{\theta_{i}}\big[Var(\theta_{i}|X_i)1\{|X_i|> \zeta_{\tau}\}\big]+E_{\theta_{i}}\big[Var(\theta_{i}|X_i)1\{|X_i|\leqslant \zeta_{\tau}\}\big]\nonumber\\
&\lesssim 1+KM \mbox{ as } \tau \rightarrow 0.\nonumber
\end{aligned}
\end{eqnarray}

Thus, for $\frac{1}{2}\leqslant a <1$, we have,
\begin{equation}
 \sum_{i:\theta_{i}\neq 0}E_{\theta_{i}} Var(\theta_{i}|X_i) \lesssim \tilde{p_n} \quad \mbox{as } \tau \rightarrow 0.\label{VAR_NONZERO_CONTRIBUTION}
\end{equation}

{\it Zero means:}\\
 
 By the bound $Var(\theta|x) \leqslant 1+x^2$ in Lemma \ref{POST_VAR_IDENTITY}, we find that,
 \begin{eqnarray}
  E_{0} Var(\theta|X)1_{\{|X| > \zeta_{\tau}\}}
  &\leqslant& 2 \int_{\zeta_{\tau}}^{\infty} (1+x^2)\frac{1}{\sqrt{2\pi}}e^{-\frac{x^2}{2}}dx\nonumber\\
  &\leqslant& \frac{\tau^{2a}}{\zeta_{\tau}} + \zeta_{\tau}\tau^{2a}.\label{VAR_THETA_0_1}
 \end{eqnarray}
 
Again when $|x| \leqslant \zeta_{\tau}$, we consider the upper bound $Var(\theta|x) \leqslant \frac{T_{\tau}(x)}{x} + J(x,\tau)$ as obtained from Lemma \ref{POST_VAR_IDENTITY}, where the term $J(x,\tau)$ is already defined in Lemma \ref{LEMMA_IMPORTANT_INTEGRAL}. Note that $\frac{T_{\tau}(x)}{x}=E(1-\kappa|x,\tau)$. Therefore, using Lemma \ref{LEM_MOMENT_INEQ}, it follows that:
\begin{equation}
\int_{-\zeta_{\tau}}^{\zeta_{\tau}} \frac{T_{\tau}(x)}{x}\frac{1}{\sqrt{2\pi}}e^{-\frac{x^2}{2}}dx 
\lesssim \tau^{2a}\zeta_{\tau}.\label{VAR_THETA_0_2}
\end{equation}

Similarly, using Lemma \ref{LEMMA_IMPORTANT_INTEGRAL} we obtain,
\begin{equation}
\int_{-\zeta_{\tau}}^{\zeta_{\tau}} J(x,\tau)\frac{1}{\sqrt{2\pi}}e^{-\frac{x^2}{2}}dx 
\lesssim \tau^{2a}\zeta_{\tau}.\label{VAR_THETA_0_3}
\end{equation}

Therefore, using (\ref{VAR_THETA_0_1}), (\ref{VAR_THETA_0_2}) and (\ref{VAR_THETA_0_3}), we have:
\begin{equation}
 \sum_{i:\theta_{i}=0} E_{0} Var(\theta_i|X_i) \lesssim (n-\tilde{p_n})\tau^{2a}\zeta_{\tau}.\label{VAR_ZERO_CONTRIBUTION} 
\end{equation}
Combining equations (\ref{VAR_NONZERO_CONTRIBUTION}) and (\ref{VAR_ZERO_CONTRIBUTION}), and taking supremum over $\theta_{0}\in l_{0}[p_{n}]$, it therefore follows that:
\begin{equation}
\sup_{\theta_{0}\in l_{0}[p_{n}]}E_{\theta_{0}} \sum_{i=1}^{n} Var(\theta_{0i}|X_i) \lesssim
  \tilde{p_n}+(n-\tilde{p_n})\tau^{2a}\sqrt{\log\big(\frac{1}{\tau^{2a}}\big)}.\nonumber
 \end{equation}
The result then follows immediately by noting that $\tilde{p_n}\leqslant p_n$ and $p_n=o(n)$.

\end{proof}

\begin{lem}\label{LEMMA_POST_VAR_LOWER_BOUND}
 Suppose the function $L(\cdot)$ given by (\ref{LAMBDA_PRIOR}) satisfies Assumption \ref{ASSUMPTION_LAMBDA_PRIOR} and is non-decreasing over $(0,\infty)$, with $a=\frac{1}{2}$. Let us define for fixed $y > 0$ and for fixed $k > 0$,
 \begin{equation}
 I_{k}= \int_{0}^{\infty} \frac{(t\tau^2)^{k-\frac{1}{2}}}{(1+t\tau^2)^{k}}t^{-3/2}L(t)e^{\frac{t\tau^2}{1+t\tau^2}y}dt. \nonumber
 \end{equation}
 Then,
 \begin{align*}
  I_{\frac{5}{2}} &\geqslant L(1)\tau\bigg[\frac{\tau}{y}\big(e^{y/2} - e^{\tau^2y}\big)+\frac{1}{\sqrt{2}y}\big(e^{y}-e^{y/2}\big) \bigg], \mbox{ for } \tau < \frac{1}{\sqrt{2}}\nonumber\\
  I_{\frac{1}{2}}  &\leqslant \tau\bigg[\frac{e^{\tau^2y}}{K\tau}+2Me^{\tau y}\big(\frac{1}{\tau}-\frac{1}{\sqrt{\tau}}\big)+2Me^{\frac{y}{2}}\big(\frac{1}{\sqrt{\tau}} - \sqrt{2} \big)+ \frac{2M\sqrt{2}}{y}\big(e^{y} - e^{\frac{y}{2}}\big)\bigg], \mbox{ for } \tau < \frac{1}{2}\nonumber\\
  I_{\frac{3}{2}} &\leqslant M\tau\bigg[e^{\tau^2y}\tau +2e^{\frac{y}{2}}\big(\frac{1}{\sqrt{2}}-\tau \big)+\frac{\sqrt{2}}{y}\big(e^{y}-e^{\frac{y}{2}}\big)\bigg], \mbox{ for } \tau < \frac{1}{\sqrt{2}}\nonumber\\
  I_{\frac{1}{2}} &\geqslant L(1)\tau\bigg[e^{\tau^2y}\big(\frac{1}{\tau}-\frac{1}{\sqrt{\tau}}\big)+\frac{\sqrt{2}}{y}\big(e^{y}-e^{\tau y}\big)+\frac{1}{2y}\big(e^{y}-e^{\frac{y}{2}}\big)\bigg], \mbox{ for } \tau < \frac{1}{2}.\nonumber
 \end{align*}
 \end{lem}
 \begin{proof}
 Note that since $L$ is nondecreasing over $(0,\infty)$, $L(t) \geqslant L(1)$ for all $t \geqslant 1$. Therefore,
 \begin{eqnarray}
  I_{\frac{5}{2}}
  &=& \int_{0}^{\infty} \frac{(t\tau^2)^2}{(1+t\tau^2)^{5/2}}t^{-3/2}L(t)e^{\frac{t\tau^2}{1+t\tau^2}y}dt\nonumber\\
  &=&\frac{1}{\tau}\int_{0}^{\infty}\big(\frac{t\tau^2}{1+t\tau^2×}\big)^{5/2}t^{-2}L(t)e^{\frac{t\tau^2}{1+t\tau^2}y}dt\nonumber\\
  &\geqslant& \frac{L(1)}{\tau}\int_{1}^{\infty}\big(\frac{t\tau^2}{1+t\tau^2×}\big)^{5/2}t^{-2}e^{\frac{t\tau^2}{1+t\tau^2}y}dt\label{INEQ_1}
 \end{eqnarray}
 Now putting $u=\frac{t\tau^2}{1+t\tau^2}$ in (\ref{INEQ_1}) we obtain,
 \begin{align*}
  I_{\frac{5}{2}}
  &\geqslant L(1)\tau \int_{\frac{\tau^2}{1+\tau^2}}^{1}u^{1/2}e^{uy}du\nonumber\\
  &\geqslant L(1)\tau \int_{\tau^2}^{1}u^{1/2}e^{uy}du [\mbox{since } \tau^2 < 1 ]\nonumber\\
  &=L(1)\tau\bigg[\frac{\tau}{y}\big(e^{y/2} - e^{\tau^2y}\big)+\frac{1}{\sqrt{2}y}\big(e^{y}-e^{y/2}\big) \bigg], \mbox{ for } \tau < \frac{1}{\sqrt{2}}\nonumber
 \end{align*}
where the last equality follows using the same set arguments in the proof of Lemma A.1 of \citet{PKV2014}.\\
  
Next observe that
\begin{eqnarray}
 I_{\frac{1}{2}}
 &=& \int_{0}^{\infty} \frac{1}{(1+t\tau^2)^{1/2}}t^{-3/2}L(t)e^{\frac{t\tau^2}{1+t\tau^2}y}dt\nonumber\\
 &=& \frac{1}{\tau}\int_{0}^{\infty}\big(\frac{t\tau^2}{1+t\tau^2}\big)^{1/2}t^{-2}L(t)e^{\frac{t\tau^2}{1+t\tau^2}y}dt\nonumber\\
 &=& \tau \int_{0}^{1}u^{-3/2}L\big(\frac{1}{\tau^2}\frac{u}{1-u}\big)e^{uy}du\mbox{ }[\mbox{putting }u=\frac{t\tau^2}{1+t\tau^2}] \nonumber\\
 &=& \tau \bigg[\int_{0}^{\tau^2}u^{-3/2}L\big(\frac{1}{\tau^2}\frac{u}{1-u}\big)e^{uy}du+\int_{\tau^2}^{1}u^{-3/2}L\big(\frac{1}{\tau^2}\frac{u}{1-u}\big)e^{uy}du\bigg]\label{INEQ_2}
\end{eqnarray}
Now observe that $e^{uy} \leqslant e^{\tau^2 y}$ for all $u \leqslant \tau^2$. Using this fact and applying the change of variable $t=\frac{1}{\tau^2}\frac{u}{1-u}$ in the first integral on the right hand side of (\ref{INEQ_2}) we obtain,
\begin{eqnarray}
 \int_{0}^{\tau^2}u^{-3/2}L\big(\frac{1}{\tau^2}\frac{u}{1-u}\big)e^{uy}du
 &\leqslant& e^{\tau^2 y} \int_{0}^{\tau^2}u^{-3/2}L\big(\frac{1}{\tau^2}\frac{u}{1-u}\big)du\nonumber\\
 &=& \frac{e^{\tau^2 y}}{\tau} \int_{0}^{\frac{1}{1-\tau^2}}\frac{t^{-3/2}}{\sqrt{1+t\tau^2}}L(t)dt\nonumber\\
 &\leqslant& \frac{e^{\tau^2 y}}{\tau} \int_{0}^{\infty}t^{-3/2}L(t)dt \mbox{ } [\mbox{since }\frac{1}{\sqrt{1+t\tau^2}} \leqslant 1]\nonumber\\
 &=& \frac{K^{-1}e^{\tau^2 y}}{\tau}\label{INEQ_3}
\end{eqnarray}
For the second integral on the right hand side of (\ref{INEQ_2}), we observe that the function $L$ is bounded by the constant $M > 0$. Using this observation and then apply the same arguments given in the proof of Lemma A.1 of \citet{PKV2014}, we obtain,
\begin{equation}
 \int_{\tau^2}^{1}u^{-3/2}L\big(\frac{1}{\tau^2}\frac{u}{1-u}\big)e^{uy}du
\leqslant 2M \bigg[e^{\tau y}\big(\frac{1}{\tau}-\frac{1}{\sqrt{\tau}}\big)+ e^{\frac{y}{2}}\big(\frac{1}{\sqrt{\tau}} - \sqrt{2}\big) + \frac{\sqrt{2}}{y}\big(e^{y} - e^{\frac{y}{2}}\big)\bigg] \mbox{} \mbox{ for } \tau < \frac{1}{2}\label{INEQ_4}.
\end{equation}
(\ref{INEQ_2}), (\ref{INEQ_3}) and (\ref{INEQ_4}) together immediately give the stated upper bound on $I_{\frac{1}{2}}$.\\

Again note that $ \tau^2 < u < 1 \Longrightarrow \frac{1}{\tau^2}\frac{u}{1-u} > \frac{1}{1-\tau^2} $ which is strictly greater than 1. Hence $L\big(\frac{1}{\tau^2}\frac{u}{1-u}\big) \geqslant L\big(\frac{1}{1-\tau^2}\big) \geqslant L(1) $ as $L$ is nondecreasing. Using this observation and (\ref{INEQ_2}) and then applying the same reasoning given in the proof of Lemma A.1 of \citet{PKV2014}, we obtain,
\begin{eqnarray}
 I_{\frac{1}{2}}
 &\geqslant& \tau \int_{\tau^2}^{1}u^{-3/2}L\big(\frac{1}{\tau^2}\frac{u}{1-u}\big)e^{uy}du\nonumber\\
 &\geqslant& L(1)\tau \int_{\tau^2}^{1}u^{-3/2}e^{uy}du\nonumber\\
 &=& L(1)\tau\bigg[e^{\tau^2y}\big(\frac{1}{\tau}-\frac{1}{\sqrt{\tau}}\big)+\frac{\sqrt{2}}{y}\big(e^{y}-e^{\tau y}\big)+\frac{1}{2y}\big(e^{y}-e^{\frac{y}{2}}\big)\bigg], \mbox{ for } \tau < \frac{1}{2}.\nonumber
\end{eqnarray}

The stated upper bound for $I_{3/2}$ follows immediately by noting that $L$ is bounded by the constant $M > 0$ and subsequently by change the variable $u=\frac{t\tau^2}{1+t\tau^2}$ followed by the same set of arguments given in the proof of Lemma A.1 of \citet{PKV2014}. This completes the proof of Lemma \ref{LEMMA_POST_VAR_LOWER_BOUND}.
 \end{proof}

 \paragraph{Proof of Theorem \ref{THM_3.5}}
 \begin{proof}
  From (\ref{VAR_IDNT_2}) we have,
 \begin{eqnarray}
   Var(\theta|x)
  &\geqslant& x^2E\bigg[(1-\kappa)^2 | x,\tau\bigg] - x^2E^2\bigg[(1-\kappa) | x,\tau\bigg] \nonumber\\
  &=& x^2 \bigg[\frac{\int_{0}^{\infty}\frac{(t\tau^2)^2}{(1+t\tau^2)^{5/2}}t^{-3/2}L(t)e^{-\frac{x^2}{2(1+t\tau^2)}}dt}{\int_{0}^{\infty}\frac{1}{(1+t\tau^2)^{1/2}}t^{-3/2}L(t)e^{-\frac{x^2}{2(1+t\tau^2)}}dt}-\quad \bigg(\frac{\int_{0}^{\infty}\frac{t\tau^2}{(1+t\tau^2)^{3/2}}t^{-3/2}L(t)e^{-\frac{x^2}{2(1+t\tau^2)}}dt}{\int_{0}^{\infty}\frac{1}{(1+t\tau^2)^{1/2}}t^{-3/2}L(t)e^{-\frac{x^2}{2(1+t\tau^2)}}dt}\bigg)^2 \bigg]\nonumber\\
&=& 2y \bigg[\frac{I_{\frac{5}{2}}}{I_{\frac{1}{2}}} - \bigg(\frac{I_{\frac{3}{2}}}{I_{\frac{1}{2}}}\bigg)^2 \bigg]\nonumber
\end{eqnarray}
 by putting $y=\frac{x^2}{2}$ and rest of the proof follows by applying Lemma \ref{LEMMA_POST_VAR_LOWER_BOUND} and the same set of arguments given in the proof of Theorem 3.4 of \citet{PKV2014}.  
 \end{proof}
 
\section*{Acknowledgements}
 The authors would like to thank Professor Jayanta Kumar Ghosh for letting them aware about the recent work of \citet{PKV2014} on the posterior contraction properties of the horseshoe prior.

\bibliographystyle{apalike}
\bibliography{reference_contraction}
\end{document}